\documentclass[12pt]{amsart} %
\usepackage{fullpage}
\usepackage{graphicx}
\usepackage{graphics}
\usepackage{amsmath}
\usepackage{amsfonts}
\usepackage{amssymb}
\usepackage{amsthm}
\usepackage{psfrag}
\usepackage{amsmath,amssymb,color}
\usepackage{amsfonts,amstext}
\usepackage{enumerate}

\newtheorem{theorem}{Theorem}

\newtheorem{conjecture}{Conjecture}
\newtheorem{lemma}{Lemma}
\newtheorem{proposition}{Proposition}

\title{A relaxation of Steinberg's Conjecture\thanks{Research supported in part by the NSA grant H98230-12-1-0226 and a NSF CSUMS grant}}
\date{\today}

\author{Owen Hill\and Gexin Yu}

\address{ Department of Mathematics, College of William and Mary, Williamsburg, VA 23185.}\email{oshill@email.wm.edu,  gyu@wm.edu}

\begin{document}
\maketitle

\begin{abstract}
A graph is {\em $(c_1, c_2, \cdots, c_k)$-colorable} if the vertex set can be partitioned into $k$ sets $V_1,V_2, \ldots, V_k$, such that for every $i: 1\leq i\leq k$ the subgraph $G[V_i]$ has maximum degree at most $c_i$.  We show that every planar graph without $4$- and $5$-cycles is $(1, 1, 0)$-colorable and $(3,0,0)$-colorable.  This is a relaxation of the Steinberg Conjecture that every planar graph without $4$- and $5$-cycles are properly $3$-colorable (i.e., $(0,0,0)$-colorable).
\end{abstract}

\section{Introduction}



It is well-known that the problem of deciding whether a planar graph is properly $3$-colorable is NP-complete.  Gr\"{o}tzsch in 1959~\cite{G59} showed the famous theorem that every triangle-free planar graph is $3$-colorable.   A lot of research was devoted to find sufficient conditions for a planar graph to be $3$-colorable, by allowing a triangle together with some other conditions.  One of such efforts is the following famous conjecture made by Steinberg in 1976.

\begin{conjecture}[Steinberg, \cite{S76}]
All planar graphs without $4$-cycles and $5$-cycles are $3$-colorable.
\end{conjecture}

Not much progress in this direction was made until Erd\"{o}s proposed to find a constant $C$ such that a planar graph without cycles of length from $4$ to $C$ is $3$-colorable.  Borodin, Glebov, Raspaud, and Salavatipour~\cite{BGRS05} showed that $C\le 7$.  For more results, see the recent nice survey by Borodin~\cite{B12}.

Yet another direction of relaxation of the Conjecture is to allow some defects in the color classes.  A graph is {\em $(c_1, c_2, \cdots, c_k)$-colorable} if the vertex set can be partitioned into $k$ sets $V_1,V_2, \ldots, V_k$, such that for every $i: 1\leq i\leq k$ the subgraph $G[V_i]$ has maximum degree at most $c_i$.    Thus a $(0,0,0)$-colorable graph is properly $3$-colorable.

Eaton and Hull~\cite{EH99} and independently \v{S}krekovski~\cite{S99} showed that every planar graph is $(2,2,2)$-colorable (actually choosable).     Xu~\cite{X08} proved that all planar graphs with no adjacent triangles or $5$-cycles are $(1,1,1)$-colorable.  Chang, Havet, Montassier, and Raspaud~\cite{CHMR11} proved that all planar graphs without $4$-cycles or $5$-cycles are $(2,1,0)$-colorable and $(4,0,0)$-colorable.  In this paper, we further prove the following relaxation of the Steinberg Conjecture.


\begin{theorem}\label{110-coloring}
  All planar graphs without $4$-cycles and $5$-cycles are $(1,1,0)$-colorable.
\end{theorem}

\begin{theorem}\label{300-coloring}
All planar graphs without $4$-cycles and $5$-cycles are $(3,0,0)$-colorable.
\end{theorem}



We will use a discharging argument in the proofs.  We let the initial charge of vertex $u\in G$ be $\mu(u)=2d(u)-6$, and the initial charge of face $f$ be $\mu(f)=d(f)-6$.  Then by Euler's formula, we have 
\begin{equation}\label{euler}
\sum_{v\in V(G)} \mu(u)+\sum_{f\in F(G)} \mu(f)=-12.
\end{equation}

Our goal is to show that we may re-distribute the charges among vertices and faces so the final charges of the vertices and faces are non-negative, which would be a contradiction.  In the process of discharging, we will see that some configurations prevent us from showing some vertices or faces to have non-negative charges. Those configurations will be shown to be reducible configurations, that is, a valid coloring outside of the configurations can be extended to the whole graph.  It is worth to note that in the proof of Theorem~\ref{110-coloring},  we prove a somewhat global structure, a special chain of triangles, to be reducible.

The following are some simple observations about the minimal counterexamples to the above theorems.

\begin{proposition}\label{fact}
Among all planar graphs without $4$-cycles and $5$-cycles that are not $(1,1, 0)$-colorable or $(3,0,0)$-colorable, let $G$ be one with minimum number of vertices.  Then\\
(a) $G$ contains no $2^-$ vertices.\\
(b) a $k$-vertex in $G$ can have $\alpha\le \lfloor \frac{k}{2} \rfloor$ incident $3$-faces, and at most $k-2\alpha$ pendant $3$-faces.
\end{proposition}

We will use the following notations in the proofs.   A {\em $k$-vertex} ($k^+$-$vertex$, $k^-$-vertex) is a vertex of degree $k$ (at least $k$, at most $k$ resp.). The same notation will apply to faces.  An {\em $(\ell_1, \ell_2, \ldots, \ell_k)$-face} is a $k$-face with incident vertices of degree $\ell_1, \ell_2, \ldots, \ell_k$. A {\em bad $3$-vertex} is a $3$-vertex on a $3$-face.  A face $f$ is a {\em pendant $3$-face} to vertex $v$ if  $v$ is adjacent to some bad $3$-vertex on $f$.  The {\em pendant neighbor} of a $3$-vertex $v$ on a $3$-face is the neighbor of $v$ not on the $3$-face.   A vertex $v$ is {\em properly colored} if all neighbors of $v$ have different colors from $v$.  A vertex $v$ is {\em nicely colored} if it shares colors with at most $\max\{s_i-1, 0\}$ neighbors, thus if a vertex $v$ is nicely colored by a color $c$ which allows deficiency $s_i>0$, then an uncolored neighbor of $v$ can be colored by $c$.

In the next section, we will give a proof to Theorem~\ref{110-coloring}; and in the last section, we will give a proof to Theorem~\ref{300-coloring}.

\section{$(1,1,0)$-coloring of planar graphs}

We will use a discharging argument in our proof.    First we will prove some reducible configurations.

Let $G$ be a minimum counterexample to Theorem~\ref{110-coloring}, that is, $G$ is a planar graph without $4$-cycles and $5$-cycles, and $G$ is not $(1,1,0)$-colorable, but any proper subgraph of $G$ is $(1,1,0)$-colorable.

The following is a very useful tool in the proofs.

\begin{lemma}\label{extending-lemma}
 Let $H$ be a proper subgraph of $G$ so that there is a $(1,1,0)$-coloring of $G-H$.   If vertex $v\in H$ satisfies either (i) $3$ neighbors of $v$ are colored, with at least two properly colored, or (ii) $4$ neighbors of $v$ are colored, all properly, then the coloring of $G-H$ can be extended to $G-(H-v)$.
 \end{lemma}

\begin{proof}

(i) Let $v\in H$ be a vertex with $3$ colored neighbors, two of which are properly colored, such that the coloring of $G-H$ can not be extended to $v$.  Since $v$ is not $(1,1, 0)$-colorable, the three neighbors of $v$ must have different colors, and furthermore, two of the colored neighbors cannot be properly colored, a contradiction to the assumption that two of the colored neighbors of $v$ are properly colored.

(ii) Let $v\in H$ be a vertex of degree $4$ with all neighbors properly colored such that the coloring of $G-H$ can not be extended to $v$. Then due to the coloring deficiencies, $v$ must have at least $2$ neighbors colored by $1$, at least $2$ neighbors colored by $2$, and at least $1$ neighbor colored by $1$.  Then $v$ has at least five colored neighbors, a contradiction.
\end{proof}

\begin{lemma}\label{334-face}
There is no $(3,3,4^-)$-face in $G$.
\end{lemma}

\begin{proof}
Let $uvw$ be a $(3,3,4^-)$-face in $G$ with $d(u)=d(v)=3$ and $d(w)\le 4$.  Then $G$\textbackslash$\{u,v,w\}$ is $(1,1,0)$-colorable.  Color $w$ and $v$ properly, then $u$ is colorable by Lemma~\ref{extending-lemma},  thus $G$ is $(1,1,0)$-colorable,  a contradiction.
\end{proof}

\begin{lemma}\label{5-vertex}
There is no $5$-vertex that is incident to two $(3,4^-,5)$-faces and adjacent to a $3$-vertex in $G$.
\end{lemma}

\begin{center}
\begin{figure}[ht]
\includegraphics[scale=1]{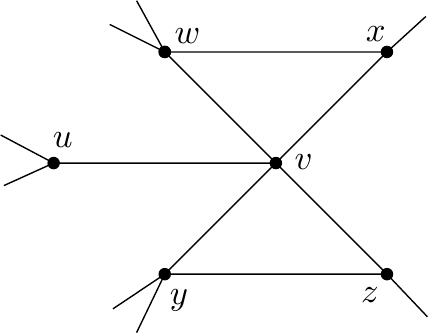}
\caption{Figure for Lemma~\ref{5-vertex}}
\label{f1}
\end{figure}\end{center}

\begin{proof}
Let $v$ be a $5$-vertex with neighbors $u, w, x, y, z$ so that $wx, yz\in E(G)$ and $d(u)=d(x)=d(z)=3$ and $d(w), d(y)\le 4$ (See Figure~\ref{f1}).  By the minimality of $G$, $G$\textbackslash $\{u,v,w,x,y,z\}$ is $(1,1,0)$-colorable.  Properly color $u$, $w$, and $y$, then properly color $x$ and $z$.  For $v$ to not be colorable, $v$ must have two neighbors colored by $1$, two neighbors colored by $2$ and one neighbor colored by $3$.  Since the $w,x$ and $y,z$ vertex pairs must be colored differently, one of them must have the colors $1$ and $2$.  W.l.o.g. we can assume that $w$ is colored by $1$ and $x$ by $2$.  Then since $w$ is properly colored, we can either recolor $x$ by $1$ or $3$, and color $v$ by $2$ obtaining a coloring of $G$, a contradiction.
\end{proof}

\begin{lemma}\label{333-vertices}
No $3$-vertex in $G$ can be adjacent to two other $3$-vertices. In particular, the $3$-vertices on a $(3,3, 5^+)$-face must have another neighbor with degree four or higher.
\end{lemma}

\begin{proof}
Let $v$ be a $3$-vertex with $x$ and $y$ being two neighbors of degree $3$.  By the minimality of $G$, $G$\textbackslash $\{v,x,y\}$ is $(1,1,0)$-colorable.  Then we can first properly color $x$ and $y$, and then by Lemma~\ref{extending-lemma} color $v$ to get a coloring of $G$, a contradiction.
\end{proof}

\begin{lemma}\label{344-face}
The pendant neighbor of the $3$-vertex on a $(3,4,4)$-face must have degree $4$ or higher.
 \end{lemma}

\begin{center}\begin{figure}[ht]
\includegraphics[scale=0.8]{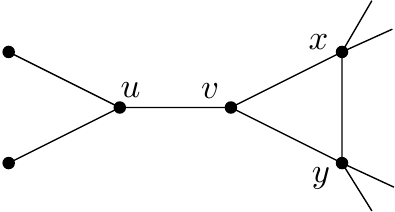} \hskip 0.4in
\includegraphics[scale=0.8]{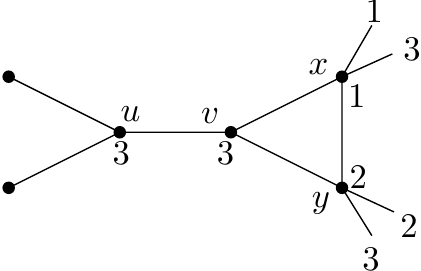}
\caption{Figure for Lemma~\ref{344-face}}
\label{fig2}
\end{figure}\end{center}

\begin{proof}
Let $vxy$ be a $(3,4,4)$-face in $G$ such that the pendant neighbor $u$ of the $3$-vertex $v$ has degree $3$ (See Figure~\ref{fig2}).  By the minimality of $G$, $G$\textbackslash $\{u,v\}$ is $(1,1,0)$-colorable.  We properly color $u$ and then color $v$ differently from both $x$ and $y$.  If $u$ and $v$ are not both colored by $3$, then we get a coloring for $G$, a contradiction,  so we may assume both $u$ and $v$ are colored by $3$.  This means that both $u$ and $v$ have two remaining neighbors colored by $1$ and $2$.  Let $x$ and $y$ be colored by $1$ and $2$ respectively.  The neighbors of $x$ must be colored by $1$ and $3$ or else we could recolor $v$ by $1$ and $x$ by $3$ if necessary to obtain a coloring of $G$.  Likewise, the neighbors of $y$ must be colored by $2$ and $3$.  In this case we switch the colors of $x$ and $y$ and color $v$ by $1$  to obtain a coloring of $G$, a contradiction again.
\end{proof}


Let a {\em $(T_0, T_1, \ldots, T_n)$-chain} be a sequence of triangles, $T_0, T_1, \ldots, T_n$, such that (i) $T_0$ is a $(3,4,4)$-face and $T_n$ is a $(3^+, 4, 4^+)$-face, and all other triangles are $(4, 4,4)$-faces, and (ii) for $0\le i\le n-1$, $T_i$ and $T_{i+1}$ share a $4$-vertex $t_i$.  In a $(T_0, T_1, \ldots, T_n)$-chain, let $x_i\in T_i$ for $0\le i\le n$ be a non-connecting $4^+$-vertex.

Let a \emph{special 4-vertex} be a $4$-vertex that is incident to one $3$-face and has two pendant $3$-faces,  and let a $3$-face be a \emph{special 3-face} if it has at least one special $4$-vertex.  Let a \emph{good 4-vertex} be a $4$-vertex with only one incident $3$-face and at most one pendant $3$-face.

We will prove in the following lemmas that a $(3,4,4)$-face $T_0$ may get help in discharging from a $(3^+, 4^+, 5^+)$-face or special $3$-face $T_n$ through a $(T_0, T_1, \ldots, T_n)$-chain.

\begin{lemma}
There are no special $(3,4,4)$-faces in $G$.
\end{lemma}

\begin{proof}
Let $uvw$ be a special $(3,4,4)$-face in $G$ such that $d(v)=d(w)=4$.  W.l.o.g. we can assume that $v$ is a special $4$-vertex with pendant neighbors $v_1$ and $v_2$.  By the minimality of $G$, $G$\textbackslash$\{u,v,v_1,v_2,w\}$ is $(1,1,0)$-colorable.  We can properly color $w$ and $u$ in that order then properly color $v_1$ and $v_2$.  Then by Lemma~\ref{extending-lemma}, we can color $u$, obtaining a coloring of $G$, a contradiction.
\end{proof}

The following is a very useful tool in extending a coloring to  a chain.

\begin{lemma}\label{chain-extending-lemma}
Consider a $(T_0, T_1, \cdots,T_n)$-chain with $n\geq 1$ and $T_n$ being a $(4,4^-,k)$-face. If $G$\textbackslash $\{T_0, T_1, \cdots, T_{n-1}\}$ has a coloring such that the $k$-vertex of $T_n$ is properly colored, or it shares the same color with the $4^-$-vertex, then the coloring can be extended to $G$.
\end{lemma}

\begin{proof}
We assume that the $(4,4^-,k)$-face $T_n$ has $k$-vertex $x_n$ and $4^-$-vertex $t_n$.  Also let $G$\textbackslash $\{T_0, T_1, \cdots, T_{i-1}\}$ has a coloring such that $x_n$ is properly colored or shares the same color with $t_n$ and $G$ does not have a $(1,1,0)$-coloring. Finally let $u$ be the $3$-vertex of $T_0$ and let $w$ be the pendant neighbor of $u$.

We consider two cases.  First let $n=1$. If $x_1$ and $t_1$ have the same color,  then we can properly color $x_0$ and $t_0$ in that order, thus by Lemma~\ref{extending-lemma} we can color $u$ so $G$ has a $(1,1,0)$-coloring, a contradiction.  So we know that $x_1$ and $t_1$ must be colored differently, and further $x_1$ is colored properly.  We can properly color $x_0$.  If $x_0$ and $w$ share the same color then we can color $t_0$ by Lemma~\ref{extending-lemma} and properly color $u$, a contradiction.  So we may assume that $x_0$ and $w$ are colored differently. If any two of $x_0, x_1,$ and $t_1$ are colored the same then we could color $t_0$ properly and color $u$ by Lemma~\ref{extending-lemma}, a contradiction.  Since $x_0, x_1,$ and $t_1$ are colored differently, if $x_0$ is not colored by $3$ then we could color $t_0$ by the same color as $x_0$ and properly color $u$, a contradiction.  So $x_0$ must be colored by $3$ and w.l.o.g. we can assume that $w$ is colored by $1$. Since $x_1$ is properly colored, it must be colored by $2$, or we could color $t_0$ by $1$ and properly color $u$, a contradiction. It follows that $t_1$ is colored by $1$.  If $t_1$ is colored properly, then we could color $t_0$ by $1$ and properly color $u$, a contradiction, so we may assume that $t_1$ is not colored properly.  Further, neither $z$ nor $z'$ (the two other neighbors of $t_1$) can be colored by $2$, or we could recolor $t_1$ properly, then color $t_0$ by $1$ and $u$ properly, a contradiction. So we color $t_1$ by $2$ and $t_0$ by $1$, and properly color $u$, a contradiction.

Now we assume that $n\geq 2$.  For all $j: 1\leq j\leq n$, properly color $x_{n-j}$ and color $t_{n-j}$ by Lemma~\ref{extending-lemma}, or properly if possible.  Then since $x_1$ was properly colored, and $t_1$ was colored after $x_1$, either $x_1$ remains properly colored, or $t_1$ has the same color as $x_1$.  Also, we know that $T_1$ must be a $(4,4,4)$-face, so by the previous case, we can extend the coloring to $T_0$ and get a coloring of $G$, a contradiction.
\end{proof}

\begin{lemma}\label{chain-to-344}
There is no $(T_0, \ldots, T_n)$-chain so that (i) $n\geq 1$ and $T_n$ is a special $(4,4,4)$-face or (ii) $n\geq 2$ and $T_n$ is a $(3,4,k)$-face or (iii) $n=1$ and $T_n$ is a $(3,4,4^-)$-face.
\end{lemma}

\begin{proof}
Let $T_0=ux_0t_0$ be a $(3,4,4)$-face with $d(u)=3$.

(i) 
Let $v$ be a special $4$-vertex of $T_n$ and let $y$ and $z$ be the neighbors of $v$ other than $t_n$ and $x_n$.  Let $S=\{t_i, x_i: 0\leq i\leq n-1\}$. By the minimality of $G$, $G\setminus (S\cup \{u,v,x_n,y,z\})$ has a $(1,1,0)$-coloring.  Properly color $x_n$, $y$ and $z$, then by Lemma~\ref{extending-lemma} color $v$.  Then, either $x_n$ remains properly colored or $v$ shares the same color, so by Lemma~\ref{chain-extending-lemma} we can extend the coloring to $\{T_0, T_1, \cdots, T_{n-1}\}$ to obtain a coloring of $G$.

(ii) 
Let $v$ be the $3$-vertex of $T_n$ and let $S=\{t_i, x_i: 0\leq i\leq n-1\}$.  By the minimality of $G$, $G\setminus(S\cup \{u,v\})$ has a $(1,1,0)$ coloring. Properly color $v$ and $x_{n-1}$.  Then by Lemma~\ref{extending-lemma}, we can color $t_{n-1}$.  Either $x_{n-1}$ remains properly colored or $t_{n-1}$ shares the same color, so by Lemma~\ref{chain-extending-lemma} we can extend the coloring to $\{T_0, T_1, \cdots, T_{n-2}\}$ to obtain a coloring of $G$.

(iii) Assume that $n=1$ and $T_n$ is a $(3,4,4)$-face with $3$-vertex $v$.  By the minimality of $G$, $G\setminus\{t_0, u, v, x_0, x_1\}$ has a $(1,1,0)$-coloring.  Properly color $x_0$ and $u$ in that order and properly color $x_1$ and $v$ in that order.  Then $t_0$ has four neighbors colored, all properly, so by Lemma~\ref{extending-lemma} we can color $t_0$ to get a coloring for $G$.
\end{proof}

\textbf{Remark:} By above lemma, a $(T_0, T_1)$-chain with $T_1$ being a $(3,4,5^+)$-face is not necessarily reducible. Let a \emph{bad $(3,4,5^+)$-face} be a $(3,4,5^+)$-face that shares a $4$-vertex with a $(3,4,4)$-face.

\begin{lemma}\label{chain-to-itself}
There is no $(T_0, \ldots, T_n)$-chain with $T_i=T_n$ for some $i\not=n$.
\end{lemma}

\begin{center}\begin{figure}[ht]
\includegraphics[scale=1]{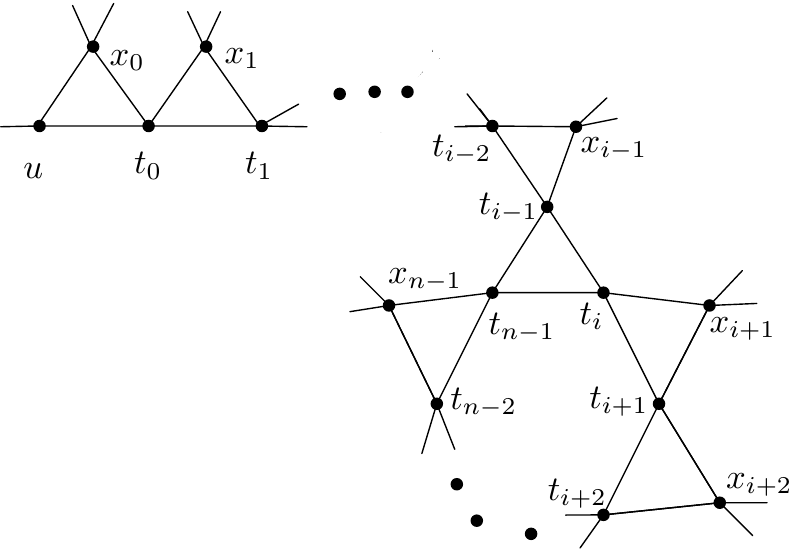}
\caption{Figure for Lemma~\ref{chain-to-itself}}
\label{fig3}
\end{figure}\end{center}

\begin{proof}
Let $(T_0, \ldots,T_n)$-chain be a chain with $T_i=T_n$ for some $i<n$. Let $u$ be the $3$-vertex of $T_0$ and let $S=\{t_j,x_j: 0\leq j\leq n-1\}$.  Since $T_i=T_n$, the vertex that would have been labelled $x_i$ is instead labelled $t_{n-1}$ (See Figure~\ref{fig3}).  By the minimality of $G$, $G\setminus (S\cup\{u\})$ is $(1,1,0)$-colorable. Start by properly coloring $x_{i+1}$, $x_{i+2}$, and $t_{i+1}$.  Then for all $j: i+2\leq j\leq n-2$, properly color $x_{j+1}$ and color $t_j$ by Lemma~\ref{extending-lemma}.  Next, properly color $t_{n-1}$, and we have two cases:

\textbf{Case 1:} $i=0$. We can properly color $u$, then color $t_i$ by Lemma~\ref{extending-lemma} to get a coloring of $G$, a contradiction.

\textbf{Case 2:} $i>0$. We can then color $t_i$ by Lemma~\ref{extending-lemma} and then either $t_{n-1}$ is properly colored, or $t_{i}$ shares the same color, so by Lemma~\ref{chain-extending-lemma} we can extend the coloring to $\{T_0, T_1, \cdots, T_{i-1}\}$ to obtain a coloring of $G$, a contradiction.
\end{proof}

\begin{lemma}\label{existence-chain}
 For each $(3,4,4)$-face $T_0$ without good $4$-vertices, there exist two chains,  $(T_0, \ldots, T_n)$-chain and $(T_0, \ldots, T_m')$-chain, such that $T_n$ and $T_m'$ are either bad $(3,4,5^+)$-faces, $(4,4^+,5^+)$-faces, or $(4,4,4)$-faces with a good $4$-vertex.   Furthermore, $T_n\not=T_m'$. 
 \end{lemma}

\begin{proof}
As $G$ is finite, any chain of triangles in $G$ must be finite.  By Lemma~\ref{chain-to-344} and ~\ref{chain-to-itself}, no chain of triangles in $G$ can end with a special $3$-face or a non-bad $(3,4,5^+)$-face, thus it must end with a bad $(3,4,5^+)$-face or a $(4,4^+,4^+)$-face.  Since a $(4,4,4)$-face in a chain can not be a special $3$-face, any chain of triangles in $G$ must end with a bad $(3,4,5^+)$-face, a $(4,4^+,5^+)$-face or a $(4,4,4)$-face with a good $4$-vertex.

Now we assume that $T_n = T_m$.  Then by Lemma~\ref{chain-extending-lemma}, $T_n$ must be a $(4,4,5^+)$-face, and since $G$ has no 4- and 5-cycles, $n+m\geq 6$. Assume that $n\leq m$.  Let $S=\{t_i, x_i: 1\leq i\leq n-1\}$, where $S=\emptyset$ if $n=1$, and $S'=\{t_j', x_j': 0\leq j\leq m-1\}$ and let $u$ be the $3$-vertex of $T_0$.  By the minimality of $G$, $G$\textbackslash $S\cup S'\cup\{u\}$ has a $(1,1,0)$-coloring. We have two cases:

If $n=1$, properly color $x_{m-1}'$ and $t_{m-1}'$. Then, by Lemma~\ref{chain-extending-lemma} we can extend the coloring to $\{T_0, T_1', \cdots T_{m-2}'\}$ to obtain a coloring of $G$, a contradiction.

If $n\geq 2$, then properly color $x_{n-1}$, $t_{n-1}$ and $x_{m-1}'$ in that order, then by Lemma~\ref{extending-lemma} we can color $t_{m-1}'$. If $n\geq 3$, for all $i: 2\leq i\leq n-1$, properly color $x_{n-i}$ and by Lemma~\ref{extending-lemma} we can color $t_{n-1}$.  Then since either $x_{m-1}'$ is still properly colored or shares the same color as $t_{m-1}'$, by Lemma~\ref{chain-extending-lemma} we can extend the coloring to $\{T_0, T_1', \cdots, T_{m-2}'\}$ to obtain a coloring of $G$, a contradiction.
\end{proof}

 We will now prove some lemmas which will ensure that bad $(3,4,5^+)$-faces will have extra charge to help $(3,4,4)$-faces.

\begin{lemma}\label{bad-345s}
A $5$-vertex incident to a bad $(3,4,5)$-face cannot be incident to another bad $(3,4,5)$-face or a $(3,3,5)$-face.
\end{lemma}

\begin{center}
\begin{figure}[ht]
\includegraphics[scale=0.7]{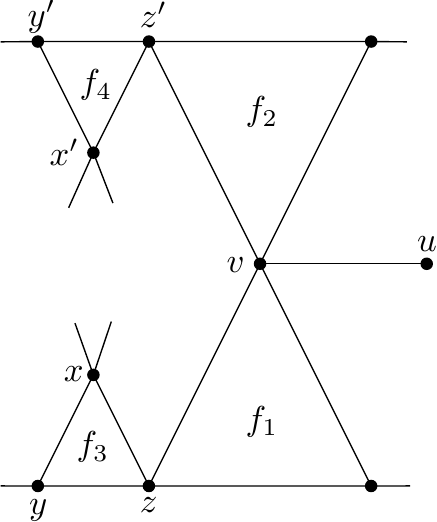}
\caption{Figure for Lemma~\ref{bad-345s}}
\label{figure4}
\end{figure}\end{center}

\begin{proof}
We only show the case when a $5$-vertex $v$ is incident to two bad $(3,4,5)$-faces, and it is very similar (and easier!) to show the case when it is incident to a bad $(3,4,5)$-face and a $(3,3,5)$-face.

Let $v$ be a $5$-vertex that is incident two bad $(3,4,5)$-faces, $f_1$ and $f_2$, and let $u$ be a $k$-vertex adjacent $v$ (see Figure~\ref{figure4}).  Let $f_3$ be the $(3,4,4)$-face sharing a $4$-vertex with $f_1$ and let $f_4$ be the $(3,4,4)$-face sharing a $4$-vertex with $f_2$.  Let $f_3$ and $f_4$ have outer $4$-vertices of $x$ and $x'$ respectively and $3$-vertices of $y$ and $y'$ respectively.  Also, let $f_1$ and $f_2$ have $4$-vertices $z$ and $z'$.  Then, by the minimality of $G$, $G$\textbackslash $\{f_1, f_2, f_3, f_4\}$ has a $(1,1,0)$-coloring.

If $u$ is colored by $1$ or $2$, then we can color $v$ by $3$ and color the $3$-vertices of $f_1$ and $f_2$ properly.  Since $v$ is properly colored, by Lemma~\ref{chain-extending-lemma} we can extend the coloring to $f_1$ and $f_3$.  Then, since $v$ is colored by $3$, it would remain properly colored, so again by Lemma~\ref{chain-extending-lemma} we can extend the coloring to $f_2$ and $f_4$ to get a coloring of $G$.

If $u$ is colored by $3$, then we properly color $x$ and $x'$ then properly color $y$ and $y'$.  We then properly color $z$ and $z'$.  If either $z$ or $z'$ is colored by $3$, then we can properly color the $3$-vertices of $f_1$ and $f_2$ and color $v$ by either $1$ or $2$ getting a coloring for $G$.  So we can assume neither is colored by $3$, and w.l.o.g. we can assume that $z$ is colored by $1$.  Then since $z$ and $z'$ are properly colored, we can color the $3$-vertices of $f_1$ and $f_2$ by either $1$ or $3$.  Then since $v$ will have at most one neighbor colored by $2$, and that neighbor colored properly, we can color $v$ by $2$ to obtain a coloring for $G$.
\end{proof}

\begin{lemma}\label{35k-pendant-neighbor}
A $(3,5,k)$-face in $G$ that is incident a $5$-vertex that is also incident to a bad $(3,4,5)$-face and a pendant  $(3,4^-,4^-)$-face will have a pendant neighbor that is a $4^+$-vertex.
\end{lemma}

\begin{center}
\begin{figure}[ht]
\includegraphics[scale=0.7]{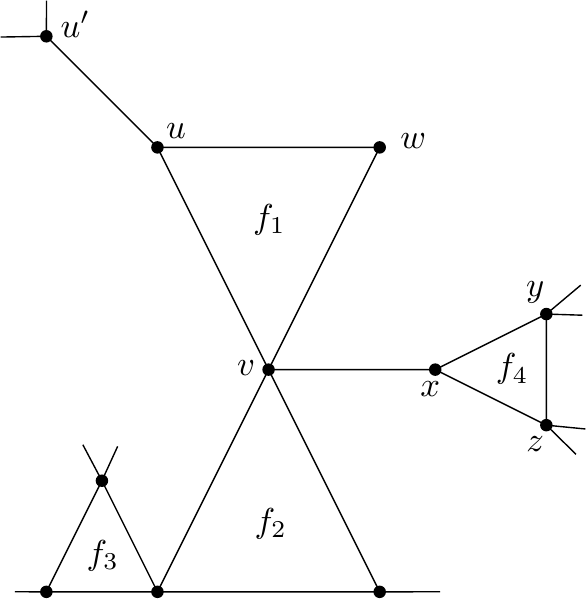}
\caption{Figure for Lemma~\ref{35k-pendant-neighbor}}
\label{figure5}
\end{figure}\end{center}

\begin{proof}
Let $f_1$ be a $(3,5,k)$-face in $G$ with a $5$-vertex $v$, a $3$-vertex $u$, and a pendant neighbor $u'$ that is a $3$-vertex.  Let the $k$-vertex of $f_1$ be $w$.  Let $v$ be incident a bad $(3,4,5)$-face $f_2$ with neighbor $(3,4,4)$-face $f_3$, and let $v$ have a pendant $(3,4,4)$-face $f_4$. Let the $3$-vertex of $f_4$ be $x$ and the $4$-vertices of $f_4$ be $y$ and $z$ (See Figure~\ref{figure5}). By the minimality of $G$, $G$\textbackslash $\{f_2, f_3, u, u', x\}$ has a $(1,1,0)$-coloring. Properly color $x$. If $w$ and $x$ share the same color, then we can properly color $u'$ and $u$, then properly color $v$ and the $3$-vertex of $f_2$.  Then the coloring can be extended to $f_3$ by Lemma~\ref{chain-extending-lemma}, obtaining a coloring of $G$. So we can assume that $w$ and $x$ are colored differently.
If $x$ is colored by $1$ or $2$ (w.l.o.g. we may assume that $x$ is colored by $1$),  then we can color $u'$ properly and color $u$ by $1$.  Then we can properly color $v$ and properly color the $3$-vertex of $f_2$.  Finally we can apply Lemma~\ref{chain-extending-lemma} to extend the coloring to $f_3$, obtaining a coloring of $G$. So we can assume that $x$ is colored by $3$.

Since $x$ is colored by $3$, we may assume that $w$ is colored by $1$.  Properly color $u'$ and color $u$ by $2$.  Since $x$ is properly colored, $y$ and $z$ must be colored by $1$ and $2$.  W.l.o.g. let $y$ be colored by $1$.  Then to avoid being able to re-color $x$ by $1$, the two other neighbors of $y$ must be colored $1$ and $3$.  For similar reasons the other two neighbors of $z$ must be colored $2$ and $3$.  Then switch the colors of $y$ and $z$ and color $x$ by $1$ or $2$ and color $v$ by $3$,  we can color the $3$-vertex of $f_2$ properly and by Lemma~\ref{chain-extending-lemma}, extend the coloring to $f_3$, obtaining a coloring of $G$.
\end{proof}

\begin{lemma}\label{355-two-bad-345}
A $(3,5,5)$-face in $G$ can not have both $5$-vertices also be incident to bad $(3,4,5)$-faces and have pendant $(3,4,4)$-faces.
\end{lemma}

\begin{center}
\begin{figure}[ht]
\includegraphics[scale=0.6]{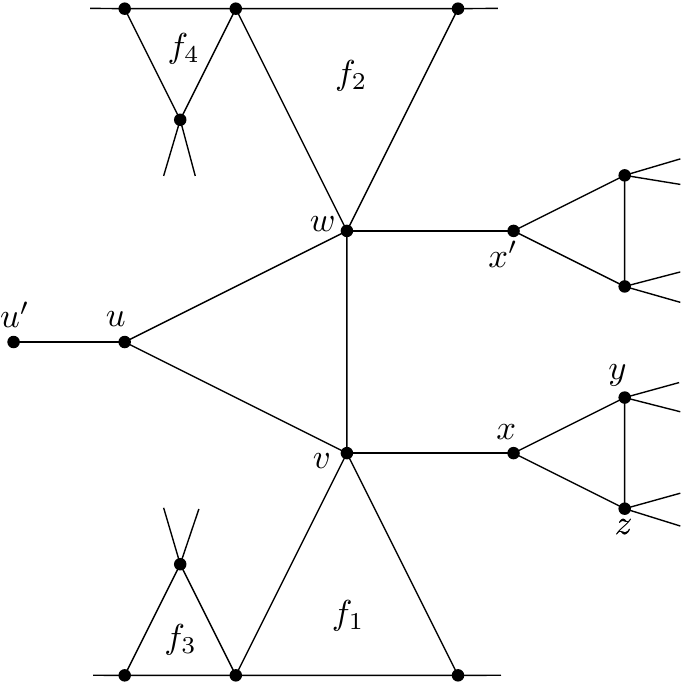}
\caption{Figure for Lemma~\ref{355-two-bad-345}}
\label{figure6}
\end{figure}\end{center}

\begin{proof}
Let $uvw$ be a $(3,5,5)$-face in $G$ where $d(v)=d(w)=5$ and $u$ has pendant neighbor $u'$.  Also let $v$ and $w$ both be incident bad $(3,4,5)$-faces, $f_1$ and $f_2$ with neighbor $(3,4,4)$-faces $f_3$ and $f_4$ respectively and let $v$ and $w$ have pendant $(3,4,4)$-faces.  Let the pendant $(3,4,4)$-faces to $v$ and $w$ have $3$-vertices $x$ and $x'$ respectively (See Figure~\ref{figure6}).  By the minimality of $G$, $G$\textbackslash $\{f_1, f_2, f_3, f_4, u, x, x'\}$ has a $(1,1,0)$-coloring.

Properly color $x$ and $x'$.  If either $x$ or $x'$ has a coloring different from $u'$, w.l.o.g. we can assume $x$, then we color $u$ the same as $x$.  We can properly color $w$ and $v$ in that order, then properly color the $3$-vertices of $f_1$ and $f_2$.  Then by Lemma~\ref{chain-extending-lemma} we can extend the coloring to $f_3$ and $f_4$ to obtain a coloring of $G$.  So we can assume that $x$, $x'$, and $u'$ are colored the same.  If $x$ is colored by $3$, since $x$ is properly colored, $y$ and $z$ must be colored by $1$ and $2$. Then to avoid being able to re-color $x$ by $1$, the other two neighbors of $y$ must be colored $1$ and $3$.  For similar reasons the other two neighbors of $z$ must be colored $2$ and $3$.  Then we can switch the colors of $y$ and $z$ and color $x$ differently from $u'$.  Then we follow the above procedure to obtain a coloring for $G$.

So we may assume that w.l.o.g. $x$, $x'$, and $u'$ are all colored by $1$. Then we color $u$ by $2$ and $w$ by $3$. Color the $3$-vertex of $f_2$ properly and by Lemma~\ref{chain-extending-lemma}, extend the coloring to $f_4$. We now have $v$ adjacent to $3$ differently and properly colored vertices.  Properly color the outer $4$-vertex and the $3$-vertex of $f_3$ in that order, then properly color the $4$-vertex of $f_1$.  If it is colored by $3$, then properly color the $3$-vertex of $f_1$ and color $v$ by either $1$ or $2$ to obtain a coloring of $G$.  If it is not colored by $3$, then w.l.o.g. we can assume that it is colored by $1$.  Then since it is properly colored, we can color the $3$-vertex of $f_1$ by either $1$ or $3$ and color $v$ by $2$, obtaining a coloring of $G$.
\end{proof}

\begin{lemma}\label{bad-345-45k-with-chain}
A $5$-vertex in $G$ that is incident a bad $(3,4,5)$-face and has a pendant $(3,4,4)$-face cannot also be incident a $(4,4^+,5)$-face $T_n$ that is in a $(T_0, \ldots, T_n)$-chain.
\end{lemma}

\begin{center}
\begin{figure}[ht]
\includegraphics[scale=0.75]{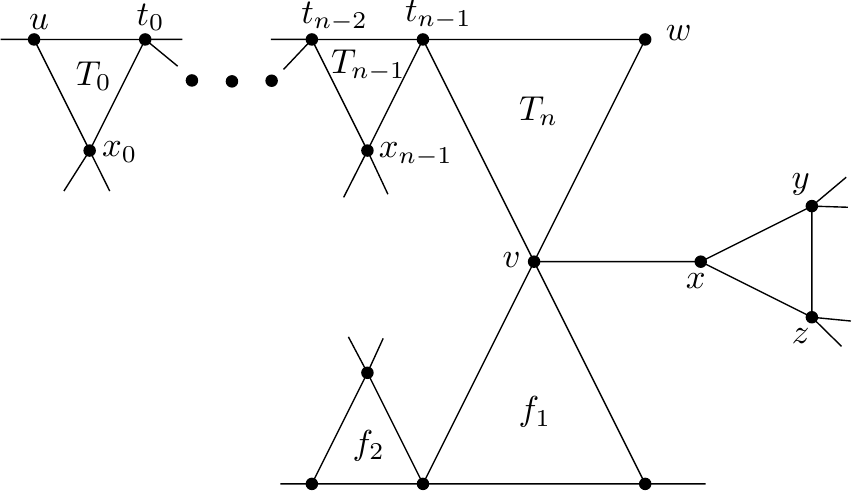}
\caption{Figure for Lemma~\ref{bad-345-45k-with-chain}}
\label{f7}
\end{figure}\end{center}

\begin{proof}
Let $v$ be a $5$-vertex in $G$ that is incident a bad $(3,4,5)$-face $f_1$ with neighbor $(3,4,4)$-face $f_2$.  Let $v$ have a pendant $(3,4,4)$-face with $3$-vertex $w$ and $4$-vertices $y$ and $z$.  Also let $v$ be incident a $(4,4^+,5)$-face $T_n$ such that there exists a chain of triangles from $T_0$ to $T_n$. Let the $4^+$-vertex of $T_n$ be $w$. Let $S=\{t_i, x_i: 0\leq i\leq n-1\}$ and let $u$ be the $3$-vertex of $T_0$ (See Figure~\ref{f7}).  By the minimality of $G$, $G\setminus(S\cup \{f_1, f_2, u, x\})$ has a $(1,1,0)$-coloring.

Properly color $x$.  If $x$ and $w$ are colored the same then we can properly color $x_{n-1}$, $t_{n-1}$, and $v$. If $n=1$, then by Lemma~\ref{extending-lemma}, we can color $u$.  If $n\geq 2$, then by Lemma~\ref{chain-extending-lemma} we can extend the coloring to $\{T_0, T_1, \cdots,T_{n-1}\}$.  Then we can properly color the $3$-vertex of $f_1$ and by Lemma~\ref{chain-extending-lemma} we can extend the coloring to $f_2$ obtaining a coloring for $G$.  So we can assume that $x$ and $w$ are colored differently.

Let $x$ be colored 1 or 2 and w.l.o.g. we can assume that $x$ is colored by $1$.  Then we can properly color $x_{n-1}$ and color $t_{n-1}$ by $1$.  Since $w$ and $x$ are colored differently, either $x_{n-1}$ and $t_{n-1}$ are both colored properly or share the same color.  If $n=1$, then either we can color $u$ properly or we can color $u$ by Lemma~\ref{extending-lemma}.  If $n\geq 2$, then by Lemma~\ref{chain-extending-lemma} we can extend the coloring to $\{T_0, T_1, \cdots, T_{n-1}\}$.  Then since $t_{n-1}$ and $x$ are colored the same we can properly color $v$ and the $3$-vertex of $f_1$.  By Lemma~\ref{chain-extending-lemma} we can extend the coloring to $f_2$ to obtain a coloring of $G$.

So let $x$ be colored by $3$ (then $w$ is colored $1$ or $2$).  Then $y$ and $z$ must be colored by $1$ and $2$, respectively. To avoid being able to re-color $x$ by $1$ or 2, the two other neighbors of $y$ must be colored $1$ and $3$ and the two other neighbors of $z$ must be colored $2$ and $3$. Then we switch the colors of $y$ and $z$ and re-color $x$ to be the same as $w$,  and proceed as above to get a coloring for $G$.
\end{proof}

\begin{lemma}\label{bad-346}
Every $6$-vertex in $G$ that is incident a bad $(3,4,6)$-face can be incident at most two $(3,4^-,6)$-faces.
\end{lemma}

\begin{center}
\begin{figure}[ht]
\includegraphics[scale=0.75]{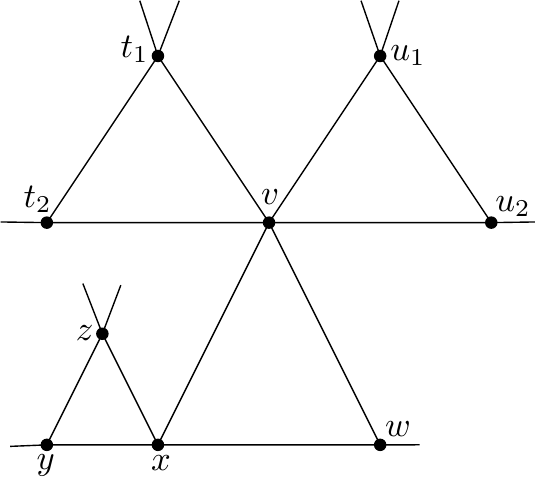}
\caption{Figure for Lemma~\ref{bad-346}}
\label{f8}
\end{figure}\end{center}

\begin{proof}
Let $v$ be a $6$-vertex in $G$. Let $vwx$ be a bad $(3,4,6)$-face with $d(w)=3$ and neighbor $(3,4,4)$-face $xyz$ with $3$-vertex $y$. Let $v$ also be incident non-bad $(3,4,6)$-faces $t_1t_2v$ and $u_1u_2v$ where $d(t_1)=d(u_1)=4$ (See Figure~\ref{f8}).  By the minimality of $G$, $G$\textbackslash $\{t_1, t_2, u_1, u_2, v, w, x, y, z\}$ has a $(1,1,0)$-coloring.  Properly color $t_1$, $t_2$, $u_1$, and $u_2$.  If the color set of $\{t_1, t_2, u_1, u_2\}$ is not $\{1,2,3\}$, then we can properly color $v$ and $w$. Then by Lemma~\ref{chain-extending-lemma}, we can extend the coloring to $x, y,$ and $z$, obtaining a coloring of $G$.  So we can assume that the color set of $\{t_1, t_2, u_1, u_2\}$ includes $1, 2,$ and $3$.

If two of $\{t_1, t_2, u_1, u_2\}$ are colored by $3$, then we can color $z$, $y$, and $x$ properly.  If $x$ is colored by $3$, then we can color $w$ properly and color $v$ by $1$ or $2$ to get a coloring of $G$. If $x$ is colored by $1$ or $2$, then since $x$ is properly colored we can color $w$ by $3$ or the same as $x$.  Then we can color $v$ differently from $3$ and $x$ to obtain a coloring of $G$.

So we can assume that exactly one of vertices in the set $\{t_1, t_2, u_1, u_2\}$ is colored by $3$.  Then w.l.o.g. we may assume that the color set of $\{t_1,t_2\}$ is $\{1,3\}$ and the color set of $\{u_1, u_2\}$ is $\{1,2\}$.  Since $u_1$ and $u_2$ were colored properly, the outside neighbor of $u_2$ must be $3$.  Let $u_1$ be colored by $1$, then since it is colored properly we can recolor $u_2$ by $1$. Then we can color $v$ and $w$ properly, and extend to $x$, $y$, and $z$ to obtain a coloring of $G$.  So we can assume that $u_1$ is colored by $2$.

Now color $z$, $y$, and $x$ properly in that order.  If $x$ is colored by $3$ then color $w$ properly.  If $w$ is colored by $1$, then color $v$ by $2$ to get a coloring for $G$.  If $w$ is colored by $2$, then since $u_1$ is colored properly recolor $u_2$ by $2$ and color $v$ by $1$ to get a coloring for $G$.  So we can assume that $x$ is colored by $1$ or $2$. Since $x$ is properly colored we can color $w$ by $3$ or the same as $x$.  Then either $1$ or $2$ but not both is in the color set of $\{x,w\}$.  If $1$ is in the color set, then $v$ will have only one neighbor colored by $2$ so we can color $v$ by $2$ and obtain a coloring of $G$.  If $2$ is in the color set, then $v$ will have two neighbors colored by $1$, but we can recolor $u_2$ by $2$ and color $v$ by $1$ to obtain a coloring of $G$.
\end{proof}

The following lemma says that a $3$-face with $k$ vertices of degree $4$ can have at most $k$ chains of triangles ending at it.

\begin{lemma}\label{at-most-two-344}
If a $(T_0, T_1, \ldots, T_n)$-chain and a $(T_0', T_1', \ldots, T_m')$-chain with $T_m'=T_n$ satisfy $T_{n-1}\cap T_n=\{t_n\}=T_{m-1}'\cap T_m'$, then $T_0=T_0'$.
\end{lemma}

\begin{proof}
For otherwise, the two chains have a common $(4,4,4)$-face $T$ so that $T=T_a$ and $T=T_b'$.  Then we would have a $(T_0, T_1, T_{a-1},T, T_{b-1}', \ldots, T_1', T_0')$-chain. But by Lemma~\ref{chain-to-344}, this chain cannot exist in $G$.
\end{proof}

\vspace{5mm}
{\large \textbf{Discharging Procedure}}

\vspace{5mm}
As we mentioned in the introduction, we set the initial charge of a vertex $v$ to be $\mu(v)=2d(v)-6$ and the initial charge of a face $f$ to be $\mu(f)=d(f)-6$.
For the discharging procedure we must introduce the notion of a bank, which serves as a temporary placeholder for charges. We set the bank with initial charge zero and will show it has a non-negative final charge.   \\

The following are the rules for discharging:

\begin{enumerate}[(R1)]
\item Each $4$-vertex gives $\frac{1}{2}$ to each pendant $3$-face and the rest to the incident $3$-faces evenly.
\item  Every $6$-vertex gives $\frac{9}{4}$ to incident bad $(3,4,6)$-faces, $2$ to other incident $(3,4^-,6)$-faces and $\frac{3}{2}$ to all other incident $3$-faces; every $7^+$-vertex gives $\frac{9}{4}$ to all incident $3$-faces.
\item Every $6^+$-vertex gives $\frac{1}{2}$ to all pendant $3$-faces.
\item Every $(4^+, 4^+, 5^+)$-face and every $(4,4,4)$-face with a good $4$-vertex give $\frac{1}{2}$ to the bank and every bad $(3,4,5^+)$-face gives $\frac{1}{4}$ to the bank.
\item  The bank gives $\frac{1}{2}$ to each $(3, 4, 4)$-face without good $4$-vertices.

\item  Every $5$-vertex gives
\begin{enumerate}[(a)]
\item $2$ to each incident $(3,3,5)$-face and $9/4$ to each incident bad $(3,4,5)$-face.
\item $7/4$ to incident non-bad $(3,4,5)$-faces when also incident a bad $(3,4,5)$-face, and gives $2$ to incident non-bad $(3,4,5)$-faces otherwise.
\item $5/4$ to incident $(3,5^+,5^+)$-faces when also incident to a bad $(3,4,5)$-face and a pendant $(3,4^-,4^-)$-face, and gives $3/2$ to incident $(3,5^+,5^+)$-faces otherwise.
\item  $3/2$ to all $(4,4^+,5)$-faces with a chain of triangles to a $(3,4,4)$-face and gives $1$ to $(4,4^+,5)$-faces otherwise.
\item $1/2$ to each pendant $(3,4^-,4^-)$-face and $(3,3,k)$-face and $1/4$ to all other pendant $3$-faces.
\end{enumerate}
\end{enumerate}

\vspace{5mm}
Let $v$ be a $k$-vertex.  By Proposition~\ref{fact}, $k\geq 3$.

\vspace{2mm}
For $k=3$,  the final charge $\mu^*(v)$ of $v$ is $\mu^*(v)=\mu(v)=0$.

\vspace{2mm}
For $k=4$,  by (R1), the final charge of $v$ is $0$.   We note that $v$ gives at least $1$ to each incident $3$-face, and gives at least $3/2$ to $3$-faces when $v$ is a good $4$-vertex.

\vspace{2mm}
For $k=5$, if $v$ has at most one incident $3$-face, then by (R6a) and (R6e), $\mu^*(v)\geq \mu(v)-\frac{9}{4}\cdot 1-\frac{1}{2}\cdot 3=1/4>0$. Let $v$ have two incident $3$-faces $f_1$ and $f_2$ and a pendant $3$-face $f_3$.

Let $f_3$ be a $(3,4^-,4^-)$-face.  When $f_1$ is a bad $(3,4,5)$-face, by Lemma~\ref{5-vertex} $f_2$ cannot be a $(3,4^-,5)$-face.  By Lemma~\ref{bad-345-45k-with-chain}, if $f_2$ is a $(4,4^+,5)$-face, then there is no chain of triangles from some $(3,4,4)$-face to $f$, so by (R6a), (R6c), (R6d), and (R6e), $\mu^*(v)\geq \mu(v)-\frac{1}{2}\cdot 1-\frac{9}{4}\cdot 1-\frac{5}{4}\cdot 1=0$. When $f_1$ is a non-bad $(3,4,5)$-face, then by Lemma~\ref{5-vertex}, $f_2$ cannot be a $(3,4^-,5)$-face, so by (R6b), (R6c), (R6d), and (R6e), $\mu^*(v)\geq \mu(v)-\frac{1}{2}\cdot 1-2\cdot 1-\frac{3}{2}\cdot 1=0$.  When neither $f_1$ nor $f_2$ are $(3,4^-,5)$-faces, by (R6c), (R6d), and (R6e), $\mu^*(v)\geq \mu(v)-\frac{1}{2}\cdot 1-\frac{3}{2}\cdot 2=\frac{1}{2}>0$.

Now let $f_3$ be a $(3,4,5)$-face.  
When $f_1$ or $f_2$ is $(3,4^-,5)$-face, by Lemma~\ref{5-vertex}, the other one cannot be a $(3,4^-,5)$-face, so by (R6b), (R6c), (R6d), and (R6e), $\mu^*(v)\geq \mu(v)-\frac{1}{4}\cdot 1-\frac{9}{4}\cdot 1-\frac{3}{2}\cdot 1=0$. When neither $f_1$ nor $f_2$ are $(3,4^-,5)$-faces, by rules (R6c), (R6d), and (R6e), $\mu^*(v)\geq \mu(v)-\frac{1}{4}\cdot 1-\frac{3}{2}\cdot 2=\frac{3}{4}>0$.

Finally, let $v$ have two incident $3$-faces $f_1$ and $f_2$, and no pendant $3$-face.  If $f_1$ is a bad $(3,4,5)$-face, then by Lemma~\ref{bad-345s}, $f_2$ cannot also be a bad $(3,4,5)$-face or a $(3,3,5)$-face.  Then by (R6), $\mu^*(v)\geq \mu(v)-\frac{9}{4}\cdot 1-\frac{7}{4}\cdot 1=0$.  If neither $f_1$ nor $f_2$ is a bad $(3,4,5)$-face, then by (R6b), (R6c), and (R6d), $\mu^*(v)\geq \mu(v)-2\cdot 2=0$.

\vspace{2mm}
For $k=6$, if $v$ is incident to at most two $3$-faces, then by (R2) and (R3), $\mu^*(v)\geq \mu(v)-\frac{9}{4}\cdot 2-\frac{1}{2}\cdot 2=\frac{1}{2}$.  So we can assume that $v$ is incident to three $3$-faces.  If $v$ is incident a bad $(3,4,6)$-face then by Lemma~\ref{bad-346} only one other incident $3$-face can be a $(3,4^-,6)$-face.  So by (R2), $\mu^*(v)\geq \mu(v)-\frac{9}{4}\cdot 2-\frac{3}{2}\cdot 1=0$.  If $v$ is not incident a bad $(3,4,6)$-face, then by (R2), $\mu^*(v)\geq \mu(v)-2\cdot 3=0$.

\vspace{2mm}
For $k\geq 7$, if $k$ is odd, then $\mu^*(v)\geq \mu(v)-\frac{k-1}{2}\cdot \frac{9}{4}-\frac{1}{2}\cdot 1=2k-6-\frac{9k-9}{8}-\frac{4}{8}=\frac{7k-43}{8}\geq \frac{3}{4}$. If $k$ is even, then $\mu^*(v)\geq \mu(v)-\frac{k}{2}\cdot \frac{9}{4}=2k-6-\frac{9k}{8}=\frac{7k-48}{8}\geq 1$.

\vspace{5mm}
Now let $f$ be a $k$-face.  Since $G$ is a simple graph, $k\geq 3$.  By the condition that there is no $4$-cycle and $5$-cycle, $k=3$ or $k\geq 6$.  Since no faces above degree $3$ are involved in the discharging procedure,  the final charge of $6^+$-face $f$ is $\mu^*(f)=\mu(f)=d(f)-6\geq 0$.

\vspace{2mm}
For $k=3$, by Lemma~\ref{334-face}, we have no $(3,3,4^-)$-faces, but we still have a few different cases:

\vspace{2mm}
\hspace{5mm} \textbf{Case 1:} Face $f$ is a $(3,3,5^+)$-face.  By Lemma~\ref{333-vertices}, $f$ will have two pendant neighbors of degree $4$ or higher.  So by (R1), (R2), (R4), and (R7), $\mu^*(f)\geq (3-6)+2\cdot 1+\frac{1}{2}\cdot 2=0$.

\vspace{2mm}
\hspace{5mm} \textbf{Case 2:} Face $f$ is a $(3,4,4)$-face. By Lemma~\ref{344-face}, $f$ will have a pendant neighbor of degree $4$ or higher.  If $f$ has a good $4$-vertex, then by (R1), $\mu^*(f)\geq \mu(f)+\frac{3}{2}\cdot 1+1\cdot 1+\frac{1}{2}\cdot 1=0$.  If $f$ has no good $4$-vertices, then by (R5), $f$ receives $1/2$ from the bank, so $\mu^*(f)=\mu(f)+1\cdot 2+\frac{1}{2}\cdot 1+\frac{1}{2}=0$.

\vspace{2mm}
\hspace{5mm} \textbf{Case 3:} Face $f$ is a bad $(3,4,5)$-face.  By (R1), (R4) and (R6a), $\mu^*(f)=\mu(f)+1\cdot 1+\frac{9}{4}\cdot 1-\frac{1}{4}\cdot 1=0$.

\vspace{2mm}
\hspace{5mm} \textbf{Case 4:} Face $f$ is a non-bad $(3,4,5)$-face.  If the $5$-vertex of $f$ is not incident a bad $(3,4,5)$-face, then by (R1) and (R6b), $\mu^*(f)=\mu(f)+1\cdot 1+2\cdot 1=0$. If the $5$-vertex of $f$ is incident a bad $(3,4,5)$-face, then by Lemma~\ref{35k-pendant-neighbor}, $f$ has a pendant neighbor of degree $4$ or higher.  So by (R1), (R6b), and (R6e), $\mu^*(f)\geq \mu(f)+1\cdot 1+\frac{7}{4}\cdot 1+\frac{1}{4}\cdot 1=0$.

\vspace{2mm}
\hspace{5mm} \textbf{Case 5:} Face $f$ is a $(3,4,6)$-face. If $f$ is a bad $(3,4,6)$-face, then by (R1), (R2), and (R4), $\mu^*(f)=\mu(f)+1\cdot 1+\frac{9}{4}\cdot 1-\frac{1}{4}\cdot 1=0$.  If $f$ is a non-bad $(3,4,6)$-face then by (R1) and (R2), $\mu^*(f)=\mu(f)+1\cdot 1+2\cdot 1=0$.

\vspace{2mm}
\hspace{5mm} \textbf{Case 6:} Face $f$ is a $(3,4,7^+)$-face.  By (R1) and (R2), $\mu^*(f)=\mu(f)+1\cdot 1+\frac{9}{4}\cdot 1=\frac{1}{4}$.

\vspace{2mm}
\hspace{5mm} \textbf{Case 7:} Face $f$ is a $(3,5,5)$-face.  If neither $5$-vertex of $f$ is also incident to a bad $(3,4,5)$-face and a pendant $(3,4^-,4^-)$-face, then by (R6c), $\mu^*(f)=\mu(f)+\frac{3}{2}\cdot 2=0$.  If one of the $5$-vertices of $f$ is also incident to a bad $(3,4,5)$-face and a pendant $(3,4^-,4^-)$-face then by Lemma~\ref{35k-pendant-neighbor}, $f$ must have a pendant neighbor of degree $4$ or higher.  In addition, by Lemma~\ref{355-two-bad-345} the other $5$-vertex of $f$ cannot have both an incident bad $(3,4,5)$-face and a pendant $(3,4^-,4^-)$-face.  So by (R6c) and (R6e), $\mu^*(f)=\mu(f)+\frac{5}{4}\cdot 1+\frac{1}{4}\cdot 1+\frac{3}{2}\cdot 1=0$.

\vspace{2mm}
\hspace{5mm} \textbf{Case 8:} Face $f$ is a $(3,5,6^+)$-face.  If the $5$-vertex of $f$ is not incident to a bad $(3,4,5)$-face and a pendant $(3,4^-,4^-)$-face then by (R2) and (R6c), $\mu^*(f)\geq \mu(f)+\frac{3}{2}\cdot 2=0$. If the $5$-vertex of $f$ has both an incident bad $(3,4,5)$-face and a pendant $(3,4^-,4^-)$-face, then by Lemma~\ref{35k-pendant-neighbor} $f$ must have a pendant neighbor of degree $4$ or higher.  So by (R2), (R6c), and (R6e), $\mu^*(f)\geq \mu(f)+\frac{5}{4}\cdot 1+\frac{1}{4}\cdot 1+\frac{3}{2}\cdot 1=0$.

\vspace{2mm}
\hspace{5mm} \textbf{Case 9:} Face $f$ is a $(3,6^+,6^+)$-face.  By (R2), $\mu^*(f)\geq \mu(f)+\frac{3}{2}\cdot 2=0$.

\vspace{2mm}
\hspace{5mm} \textbf{Case 10:} Face $f$ is a $(4,4,4)$-face. If $f$ has no good $4$-vertices then by (R1), $\mu^*(f)=\mu(f)+1\cdot 3=0$.  If $f$ has a good $4$-vertex then by (R1) and (R4), $\mu^*(f)\geq \mu(f)+1\cdot 2+\frac{3}{2}\cdot 1-\frac{1}{2}\cdot 1=0$.

\vspace{2mm}
\hspace{5mm} \textbf{Case 11:} Face $f$ is a $(4^+,4^+,5^+)$-face.  If $f$ has no chains of triangles to a $(3,4,4)$-face, then each incident vertex gives at least $1$ to $f$, so $\mu^*(f)\geq \mu(f)+1\cdot 3=0$. If $f$ has a chain of triangles to a $(3,4,4)$-face then by (R6d), at least one vertex must give $\frac{3}{2}$ to $f$, so combined with (R4),  $\mu^*(v)\geq \mu(v)+1\cdot 2+\frac{3}{2}\cdot 1-\frac{1}{2}\cdot 1=0$.

\vspace{5mm}
Finally, we show that the bank has a non-negative charge. By Lemma~\ref{existence-chain}, for each $(3,4,4)$-face without good $4$-vertices in $G$, there exist at least two chains of triangles from the $(3,4,4)$-face to a bad $(3,4,5^+)$-face, a $(4,4,4)$-face with a good $4$-vertex, or a $(4^+,4^+,5^+)$-face.  Then by Lemma~\ref{at-most-two-344}, there exist at most two chains of triangles to $(4^+,4^+,5^+)$-face from $(3,4,4)$-faces and at most one chain of triangles to a  $(3,4,5^+)$-face from $(3,4,4)$-faces. So we can see the transfer of charge from triangles with extra charge to the bank and back to $(3,4,4)$-faces is a transfer of $\frac{1}{4}$ charge over each chain of triangles.  Each $(4,4,4)$-face with a good $4$-vertex and $(4^+,4^+,5^+)$-face gives $\frac{1}{2}$ to the bank, and the bank will give at most $\frac{1}{4}\cdot 2$ to $(3,4,4)$-faces for each $(4,4,4)$-face with a good $4$-vertex or $(4^+,4^+,5^+)$-face.  Also, each bad $(3,4,5^+)$-face gives $\frac{1}{4}$ to the bank, and the bank will give at most $\frac{1}{4}\cdot 1$ to $(3,4,4)$-faces for each bad $(3,4,5^+)$-face.  Hence the bank will always have a non-negative charge.

\vspace{5mm}
This completes the discharging, showing that the final charges of all faces, vertices, and the bank are non-negative, a contradiction to \eqref{euler}.  This completes the proof of Theorem $1.1$.

\section{ $(3,0,0)$-coloring of planar graphs}

In this section, we give a proof for Theorem~\ref{300-coloring}. Our proof will again use a discharging method.   Let $G$ be a minimum counterexample to Theorem~\ref{300-coloring}, that is, $G$ is a planar graph without $4$-cycles and $5$-cycles and is not $(3,0,0)$-colorable, but any proper subgraph of $G$ is properly $(3,0,0)$-colorable.  We may assume that vertices colored by $1$ may have up to three neighbors colored by $1$.

The following is a very useful tool to extend a coloring on a subgraph of $G$ to include more vertices.

\begin{lemma}\label{extend-coloring}
 Let $H$ be a proper subgraph of $G$.  Given a $(3,0,0)$-coloring of $G-H$, if two neighbors of $v\in H$ are colored so that one is a $5^-$-vertex and the other is nicely colored, then the coloring can be extended to $G-(H-v)$ such that $v$ is nicely colored by $1$.
 \end{lemma}

\begin{proof}
Let $H$ be a subgraph of $G$ such that $G-H$ has a $(3,0,0)$-coloring.  Let $v\in H$ have neighbors $u$ and $w$ that are colored.  Let $d(u)\leq 5$ and let $w$ be nicely colored.  Color $v$ by $1$.  Since $w$ is nicely colored, if this coloring is invalid, then $u$ must be colored by $1$.  In addition, $u$ must have at least $3$ neighbors colored by $1$.  To avoid recoloring $u$ by $2$ or $3$, $u$ must have at least one neighbor of color $2$ and at least one neighbor of color $3$.  This implies that $d(u)\geq 6>5$, a contradiction.  So $v$ is colorable by $1$.  In addition, since the deficiency of color $1$ is $3$ and $v$ only has $2$ neighbors, it follows that $v$ is nicely colored.
\end{proof}

\begin{lemma}\label{3-to-6}
Every $3$-vertex in $G$ has a $6^+$-vertex as a neighbor.
\end{lemma}

\begin{proof}
Let $v$ be a vertex in $G$ such that each neighbor vertex of $v$ has degree $5$.  By the minimality of $G$, $G-v$ is $(3,0,0)$-colorable.  If two vertices in $N(v)$ share the same color, then $v$ can be properly colored, so we can assume all the neighbors of $v$ are colored differently. Let $u$ be the neighbor of $v$ that is colored by $1$.  Then $u$ must have $3$ neighbors colored by $1$ to forbid $v$ to be colored by $1$.  In addition, $u$ must have neighbors colored by $2$ and $3$ to forbid $v$ to be colored by $2$ or $3$.  Then, $u$ has at least $6$ neighbors, a contradiction.
\end{proof}

Let a $(3,3,3^+)$-face to be \emph{poor} if the pendant neighbors of the two $3$-vertices have degrees at most $5$. A $(3, 3^+, 3^+)$-face is \emph{semi-poor} if exactly one of the pendant neighbors of the $3$-vertices has degree $5$ or less.   A $3$-face is \emph{non-poor} if each $3$-vertex on it has the pendant neighbor being a $6^+$-vertex.  Finally,  a \emph{poor 3-vertex} is a $3$-vertex on a poor or semi-poor $3$-face that has a $5^-$-vertex as its pendant neighbor.

\begin{lemma}\label{336-face}
All (3,3,6$^-$)-faces in G are non-poor.
\end{lemma}

\begin{proof}
For all $(3,3,5^-)$-faces in $G$, the proof is trivial by Lemma~\ref{3-to-6}.  Let $uvw$ be a $(3,3,6)$-face in $G$ with $d(u)=d(v)=3$ such that the pendant neighbor $v'$ of $v$ has degree at most $5$.  By the minimality of $G$, $G$\textbackslash $\{u,v\}$ is $(3,0,0)$-colorable.  Properly color $u$ and color $v$ differently than both $w$ and $v'$.  Then $u$ and $v$ are both colored by $2$ or $3$, w.l.o.g. assume $2$.  This means that $u'$ and $v'$ share the same color (where $u'$ is the pendant neighbor of $u$), different from the color of $w$.

Let $w$ be colored by $1$, then to avoid being able to recolor $u$ or $v$ by $1$, $w$ must have $3$ outer neighbors colored by $1$.  Then $w$ can be recolored by $2$ or $3$ depending on the color of its fourth colored neighbor.  We recolor $w$ by $2$ or $3$ and recolor $u$ and $v$ by $1$ to get a coloring of $G$, a contradiction.

So we may assume that $w$ is colored by $3$, and that $u'$ and $v'$ are colored by $1$.  To avoid recoloring $v$ by $1$, $v'$ must have at least $3$ neighbors colored by $1$.  In addition, to avoid recoloring $v'$ by $2$ or $3$ and coloring $v$ by $1$, $v'$ must have neighbors colored by both $2$ and $3$.  This contradicts that $v'$ has degree less than $6$.
\end{proof}

\begin{lemma}\label{incident-poor}
No vertex $v\in V(G)$ can have $\lfloor \frac{d(v)}{2} \rfloor$ incident poor $3$-faces.
\end{lemma}

\begin{proof}
Let $v$ be a $k$-vertex in $G$ with $\lfloor \frac{k}{2} \rfloor$ incident poor $(3,3,k)$-faces.  Let $u_1, u_2, \cdots, u_k$ be the neighbors of $v$, and let $u_i'$ be the pendant neighbor if $u_i$ is in a poor $3$-face. Note that $d(u_i')\le 5$ and we know that all except possibly $u_k$ are in poor $3$-faces.

By the minimality of $G$, $G$\textbackslash $\{v, u_1, u_2, \cdots, u_{k-1}\}$ is $(3,0,0)$-colorable. If $d(v)$ is odd, then by Lemma~\ref{extend-coloring}, for all $i$ with $1\leq i\leq k-1$, we can color $u_i$ by $1$. Then we can properly color $v$ to get a coloring of $G$, so we can assume that $d(v)$ is even. If $d(v)$ is even, then by Lemma~\ref{extend-coloring}, for all $i$ with $1\leq i\leq k-2$, we can color $u_i$ by $2$.  Then if $u_k$ is colored by $1$ we can color $u_{k-1}$ properly and $v$ properly to get a coloring of $G$.  If $u_k$ is colored by $2$ or $3$, then it is colored properly and by Lemma~\ref{extend-coloring} we can color $u_{k-1}$ by $1$. Then we can properly color $v$ to get a coloring of $G$, a contradiction.
\end{proof}

\begin{lemma}\label{8-vertex}
 If an $8$-vertex $v$ is incident to three incident poor $(3,3,8)$-faces, then it cannot be incident to a semi-poor face, nor two pendant $3$-faces.
 \end{lemma}

\begin{proof}
Let $v$ be an $8$-vertex in $G$ with $3$ incident poor $(3,3,8)$-faces.  Let $u_1, u_2, \cdots, u_6$ be the $3$-vertices in the poor $(3,3,8)$-face and let $u'_1, u'_2, \cdots, u'_6$ be the corresponding pendant neighbors, respectively.  We know that for all $i$ with $1\leq i\leq 6$, $d(u'_i)\leq 5$.

(i)  Let $vu_7u_8$ be the incident semi-poor face with $u_7$ being the poor $3$-vertex.  Then by the minimality of $G$, $G$\textbackslash $\{v, u_1, u_2, \cdots, u_7\}$ is $(3,0,0)$-colorable.  By Lemma~\ref{extend-coloring}, $u_1, u_2, \cdots, u_6$ can be colored by $1$.  Then if $u_8$ is colored by $1$, we can properly color $u_7$ and then $v$ to get a coloring of $G$.  So we may assume that $u_8$ is not colored by $1$, in which case it is nicely colored and we may color $u_7$ with $1$ by Lemma~\ref{extend-coloring}, and then properly color $v$ to get a coloring of $G$, a contradiction.

(ii)  Let $u_7$ and $u_8$ be the bad $3$-vertices adjacent to $v$.  Then $G$\textbackslash $\{v, u_1, u_2, \cdots, u_7, u_8\}$ is $(3,0,0)$-colorable, by the minimality of $G$.    Properly color both $u_7$ and $u_8$.  If either $u_7$ or $u_8$ is colored by $1$ or both have the same color, then by Lemma~\ref{extend-coloring}, we may color $u_1, u_2, \cdots, u_6$ by $1$ and then properly color $v$.  So we may assume that $u_7$ is colored by $2$ and $u_8$ is colored by $3$.  Then we properly color $u_1, u_2, \cdots, u_6$, and it follows that for each $i$ with $1\leq i\leq 3$, $u_{2i-1}$ and $u_{2i}$ must be colored differently.  Then $v$ can have at most $3$ neighbors colored by $1$, all properly colored, so $v$ can be colored by $1$, a contradiction.
\end{proof}

\begin{lemma}\label{7-vertex}
If a $7$-vertex $v$ is incident to two poor $(3,3,7)$-faces, then it cannot be (i) incident to a semi-poor $(3,6^-,7)$-face and adjacent to a pendant $3$-face, or (ii) adjacent to three pendant $3$-faces.
\end{lemma}

\begin{proof}
Let $v$ be a $7$-vertex in $G$ with $2$ incident poor $(3,3,7)$-faces.  Let $u_1, u_2, u_3,$ and $u_4$ be the $3$-vertices on the poor $(3,3,7)$-faces and let $u'_1, u'_2, u'_3,$ and $u'_4$ be their corresponding pendant neighbors, respectively.  We know that for all $i$ with $1\leq i\leq 4$, $d(u'_i)\leq 5$.

(i)  Let $vu_5u_6$ be a semi-poor face with $u_5$ being a poor $3$-vertex and $d(u_6)\le 6$ and let $u_7$ be a bad $3$-vertex adjacent to $v$.  By the minimality of $G$, $G$\textbackslash $\{v, u_1, u_2, u_3, u_4, u_5, u_7\}$ is $(3,0,0)$-colorable.  Since at this point $u_6$ has only $4$ colored neighbors, if $u_6$ is colored by $1$ then either it is nicely colored or it can be recolored properly.  If $u_6$ is not nicely colored, then recolor $u_6$ properly.

Color $u_7$ properly.  If $u_7$ is colored by $1$, then by Lemma~\ref{extend-coloring}, we can color $u_1, u_2, \cdots, u_5$ by $1$ and then color $v$ properly, a contradiction.  So we may assume w.l.o.g. that $u_7$ is colored by $2$.  Color $u_1, u_2, \cdots, u_5$ properly.  Then, for each $i$ with $1\leq i\leq 3$, $u_{2i}$ and $u_{2i-1}$ are colored differently and nicely.  This leaves $v$ with at most $3$ neighbors colored by $1$, all nicely, so we may color $v$ by $1$ to get a coloring of $G$, a contradiction.

(ii) Let $u_5$, $u_6$, and $u_7$ be the bad $3$-vertices adjacent to $v$.  By the minimality of $G$, $G$\textbackslash $\{v, u_1, \ldots, u_7\}$ is $(3,0,0)$-colorable.  Properly color $u_5$, $u_6$, and $u_7$.  If the set $\{u_5, u_6, u_7\}$ does not contain both colors $2$ and $3$, then by Lemma~\ref{extend-coloring},  we can color $u_1$, $u_2$, $u_3$, and $u_4$ by $1$ and color $v$ properly.  So we can assume that $\{u_5, u_6, u_7\}$ contains both colors $2$ and $3$.  This implies that at most one vertex is colored by $1$.  So we properly color $u_1$, $u_2$, $u_3$, and $u_4$.  Then $v$ has at most $3$ neighbors colored by $1$, all nicely, so we can color $v$ by $1$ to get a coloring of $G$, a contradiction.
\end{proof}

\begin{lemma}\label{377-face}
Let $uvw$ be a semi-poor $(3,7,7)$-face in $G$ such that $d(v)=d(w)=7$. Then vertices $v$ and $w$ cannot both be $7$-vertices that are incident to two poor $3$-faces,  one semi-poor $(3,7,7)$-face, and adjacent to one pendant $3$-face.
\end{lemma}

\begin{center}\begin{figure}[ht]
\includegraphics[scale=0.75]{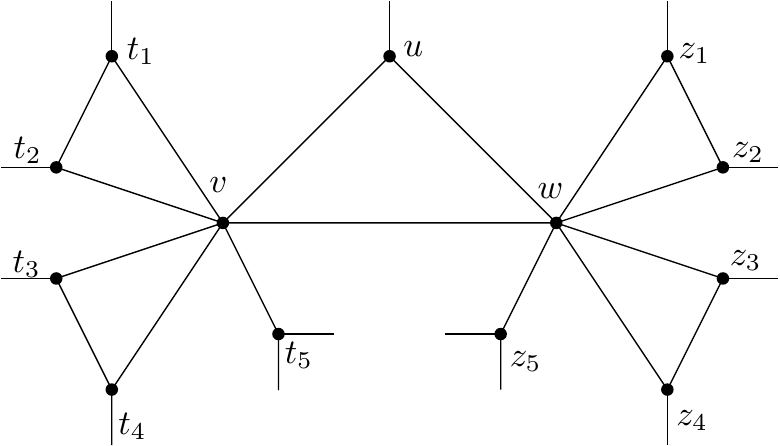} \hspace{10mm}
\includegraphics[scale=0.75]{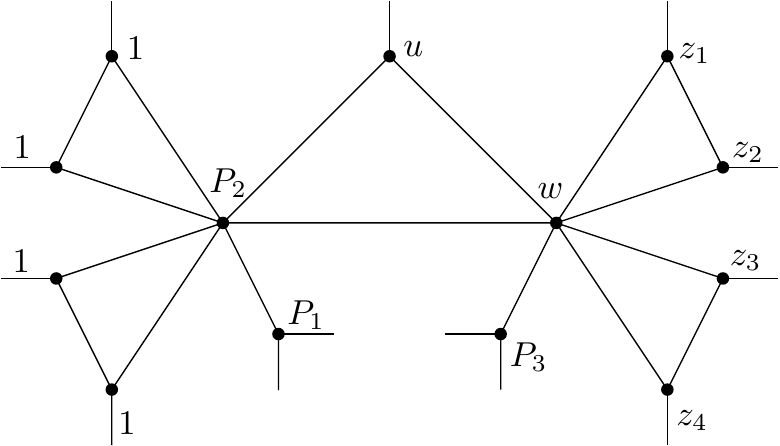}
\caption{Figure for Lemma~\ref{377-face}}
\label{fig9}
\end{figure}\end{center}

\begin{proof}
Let $uvw$ be a semi-poor $(3,7,7)$-face in $G$ such that $d(v)=d(w)=7$ and both $v$ and $w$ are incident to two poor $3$-faces, one $(3,7,7)$-face, and adjacent to one pendant $3$-face.  Let the neighbors of $v$ and $w$ be $t_1, t_2, \cdots, t_5$ and $z_1, z_2, \cdots, z_5$, respectively such that $t_5$ and $z_5$ are bad $3$-vertices (See Figure~\ref{fig9}).

By the minimality of $G$, $G$\textbackslash $\{u, v, w, t_1, t_2, \cdots, t_5, z_1, z_2, \cdots, z_5\}$ is $(3,0,0)$-colorable. By Lemma~\ref{extend-coloring}, we can color $t_1, t_2, t_3,$ and $t_4$ by $1$.  Then properly color $t_5$, $v$, and $z_5$ in that order.  Vertex $v$ will not be colored by $1$, so w.l.o.g. lets assume that $v$ is properly colored by $2$.  If $z_5$ is colored by $1$, then by Lemma~\ref{extend-coloring}, we can color $z_1, z_2, z_3, z_4,$ and $u$ by $1$ and then properly color $w$, to get a coloring of $G$, a contradiction.  So we can assume that $z_5$ is not colored by $1$.  Then we properly color $z_1, z_2, z_3, z_4$ and $u$,  so $w$ can have at most $3$ neighbors colored by $1$, all properly.  We can color $v$ by $1$ to get a coloring of $G$, a contradiction.
\end{proof}

\vspace{3mm}
{\large \textbf{Discharging Procedure:}}\\


We start the discharging process now. Recall that the initial charge for a vertex $v$ is $\mu(v)=2d(v)-6$ and the initial charge for a face $f$ is $\mu(f)=d(f)-6$.  \\

We introduce the following discharging rules:

\begin{enumerate}[(R1)]
\item Every $4$-vertex gives $1$ to each incident $3$-face.
\item Every $5$ and $6$-vertex gives $2$ to each incident $3$-face.
\item every $6^+$-vertex gives $1$ to each adjacent pendant $3$-face.
\item Each $d$-vertex with $7\le d\le 10$ gives $3$ to each incident poor $(3, 3, *)$-face, $2$ to each incident semi-poor $3$-face, except $7$-vertices give $1$ to special semi-poor $3$-face, where a special semi-poor $(3,7,7+)$-face is a semi-poor $3$-face incident to a $7$-vertex which is also incident to two poor $3$-faces and adjacent to one pendant $3$-face.  Each $d$-vertex with $7\le d\le 10$ gives $1$ to all other incident $3$-faces.
\item Every $11^+$-vertex gives $3$ to all incident $3$-faces.
\end{enumerate}

\vspace{5mm}
Now let $v$ be a $k$-vertex.  By Proposition~\ref{fact}, $k\geq 3$.

When $k=3$, $v$ is not involved in the discharging process, so $\mu^*(v)=\mu(v)=0$.

When $k=4$, by Proposition~\ref{fact}, $v$ can have at most $2$ incident $3$-faces.  By (R1), $\mu^*(v)\geq \mu(v)-1\cdot 2=0$.

When $k=5$, by Proposition~\ref{fact}, $v$ can have at most $2$ incident $3$-faces.  By (R2), $\mu$*$(v)\geq \mu(v)-2\cdot 2=0$.

When $k=6$, by Proposition~\ref{fact}, $v$ can have $\alpha\le 3$ incident $3$-faces, and at most $(k-2\alpha)$ pendant $3$-faces.  By (R2) and (R3), $\mu^*(v)\geq \mu(v)-2\cdot \alpha -1\cdot (k-2\alpha)=k-6=0$.

When $k=7$, $v$ has an initial charge $\mu(v)=7\cdot 2-6=8$.  By Lemma~\ref{incident-poor}, $v$ has at most two poor $3$-faces.  If $v$ has less than two incident poor $3$-faces, then by (R3) and (R4), $\mu$*$(v)\geq \mu(v)-3\cdot 1-1\cdot 5=0$ since $v$ gives at most one charge per vertex excluding vertices in poor $3$-faces.
So assume that $v$ has exactly $2$ incident poor $3$-faces.  By Lemma~\ref{7-vertex}, $v$ is adjacent to at most two pendant $3$-faces, and if it  is incident to a semi-poor $(3,6^-,7)$-face, then $v$ is not adjacent to a pendant $3$-face. So if $v$ is not incident to a semi-poor $(3,7^+, 7)$-face, then by (R3) and (R4), $\mu^*(v)\ge \mu(v)-3\cdot 2-2\cdot 1=0$;  If $v$ is incident to a semi-poor $(3,7^+,7)$-face, then by rules (R3) and (R4), $\mu^*(v)\geq \mu(v)-3\cdot 2-1\cdot 1-1\cdot 1=0$.

When $k=8$, $v$ has an initial charge $\mu(v)=8\cdot 2-6=10$. By Lemma~\ref{incident-poor}, $v$ has at most three poor $3$-faces. If $v$ has less than $3$ incident poor $3$-faces, then by (R3) and (R4), $\mu^*(v)\geq \mu(v)-3\cdot 2-1\cdot 4=10-6-4=0$ since $v$ gives at most one charge per vertex excluding vertices in poor $3$-faces.
So let $v$ is incident to exactly $3$ poor $3$-faces.  By Lemma~\ref{8-vertex}, $v$ cannot be incident to a semi-poor $3$-face or adjacent to two pendant $3$-faces, then  $\mu^*(v)\ge \mu(v)-3\cdot 3-1\cdot 1=0$.

When $k=9$, by Lemma~\ref{incident-poor}, $v$ is incident to at most three poor $3$-faces.  The worst case occurs when $v$ is incident $3$ poor $(3,3,9)$-faces, incident one semi-poor $(3,3,9)$-face, and pendant one $3$-face.  So by (R3) and (R4), $\mu$*$(v)\geq \mu(v)-1\cdot 1-3\cdot 3-2\cdot 1=12-1-9-2=0$.

When $k=10$, by Lemma~\ref{incident-poor}, $v$ is incident to at most four poor $(3,3,10)$-faces.  So by (R3) and (R4), $\mu^*(v)\geq \mu(v)-3\cdot 4-2\cdot 1=14-3\cdot 4-2\cdot 1=0$.

When $k\geq 11$, we assume that $v$ is incident to $\alpha$ $3$-faces, then by Proposition~\ref{fact}, $\alpha\le \lfloor k/2\rfloor$.  Thus the final charge of $v$ is $\mu^*\ge 2k-6-3\alpha-1\cdot (k-2\alpha)=k-\alpha-6\ge 0$.

\vspace{5mm}
Now let $f$ be a $k$-face in $G$.  By the conditions on $G$, $k=3$ or $k\geq 6$. When $k\geq 6$, $f$ is not involved in the discharging procedure, so $\mu*(f)=\mu(f)=k-6\geq 0$.  So in the following we only consider $3$-faces.


\hspace{5mm} \textbf{Case 1:} $f$ is a $(4^+, 4^+, 4^+)$-face.  By the rules, each $4^+$-vertex on $f$ gives at least $1$ to $f$, so $\mu*(f)\geq \mu(f)+1\cdot 3=0$.

\hspace{5mm} \textbf{Case 2:} $f$ is a $(3,4^+,4^+)$-face with vertices $u,v,w$ such that $d(u)=3$.  If $u$ is not a poor $3$-vertex, then by (R2), $f$ gains $1$ from the pendant neighbor of $u$ and by the other rules, $f$ gains at least $2$ from vertices on $f$, thus $\mu^*(f)\geq \mu(f)+1\cdot 3=0$. If $u$ is a poor vertex (it follows that $f$ is a semi-poor $3$-face), then by Lemma~\ref{3-to-6}, $f$ is a $(3, 4^+, 6^+)$-face.  Since $v$ or $w$ is a $6^+$-vertex, it gives at least $2$ to $f$ unless $f$ is a special semi-poor $(3,7,7^+)$-face, and as the other is a $4^+$-vertex, it gives at least $1$ to $f$.  Therefore, if $f$ is not a special semi-poor $3$-face, then $\mu^*(f)\geq \mu(f)+2\cdot 1+1\cdot 1=0$; if $f$ is a special semi-poor $(3,7,8^+)$-face, then $f$ receives at least $2$ from the $8^+$-vertex, so $\mu^*(v)\geq \mu(v)+2\cdot 1+1\cdot 1=0$. If $f$ is a special semi-poor $(3, 7, 7)$-face so that both $v$ and $w$ are incident to two poor $3$-faces, one semi-poor $(3, 7, 7)$-face and adjacent to one pendant $3$-face, then by Lemma~\ref{377-face}, is impossible.


\hspace{5mm} \textbf{Case 3:} $f$ is a $(3,3,4^+)$-face with $4^+$-vertex $v$. If $d(v)\ge 11$, then by (R5), $\mu^*(f)\geq \mu(f)+3=0$.  So assume $d(v)\le 10$.  By Lemma~\ref{3-to-6}, if $4\le d(v)\le 6$, then each $3$-vertex has the pendant neighbor of degree $6$ or higher.  So by (R1) and (R3)  (when $d(v)=4$),  $\mu^*(f)\ge \mu(f)+1\cdot 3=0$, or by (R1) and (R2) (when $d(v)>4$),  $\mu^*(f)=\mu(f)+2\cdot 1+1\cdot 1=0$.

Let $7\leq d(v)\leq 10$. If $f$ is poor, then by (R4), $\mu^*(f)=\mu(f)+3\cdot 1=0$.  If $f$ is semi-poor, then one $3$-vertex on $f$ is adjacent to a $6^+$-vertex and thus by (R3) $f$ gains $1$ from it, together the $2$ that $f$ gains from $v$ by (R4), we have  $\mu^*(f)=\mu(f)+2\cdot 1+1\cdot 1=0$.  If $f$ is non-poor, then both $3$-vertices on $f$ are adjacent to the pendant neighbors of degrees more than $5$, thus by (R3) and (R4), $\mu^*(f)=\mu(f)+1\cdot 2+1\cdot 1=0$.

\hspace{5mm} \textbf{Case 4:} $f$ is a $(3,3,3)$-face.  By Lemma~\ref{3-to-6}, each $3$-vertex will have the pendant neighbor of degree $6$ or higher, so by (R3), $\mu^*(f)=\mu(f)+1\cdot 3=0$.

\vspace{3mm}
Since for all $x\in V\cup F$, $\mu^*(x)\geq 0$, $\sum_{v\in V}\mu^*(v) +\sum_{f\in F}\mu^*(f) \geq 0$, a contradiction.  This completes the proof of Theorem $1.2$.

\section*{Acknowledgement}
The research is supported in part by NSA grant  H98230-12-1-0226 and NSF CSUMS grant.   The authors thank Bernard Lidicky for some preliminary discussion.


\begin{thebibliography}{99}



\bibitem{B12}
O.V. Borodin,  Colorings of planar graphs: a survey.  \emph{Disc. Math.}, to appear.

\bibitem{BGRS05}
O. V. Borodin, A. N. Glebov, A. R. Raspaud, and M. R. Salavatipour. Planar graphs without cycles of length from 4 to 7 are 3-colorable. \emph{J. of Comb. Theory, Ser. B}, {\bf 93} (2005), 303--311.

\bibitem{CHMR11}
G. Chang, F. Havet, M. Montassier, and A. Raspaud, Steinberg's Conjecture and near colorings, preprint.


\bibitem{EH99}
N. Eaton and T. Hull. Defective list colorings of planar graphs. \emph{Bull. Inst. Combin. Appl.}, {\bf 25} (1999), 78--87.





\bibitem{G59}
H. Gr\"{o}tzsch,  Ein dreifarbensatz f\:{u}r dreikreisfreienetze auf der kugel. \emph{Math.-Nat.Reihe}, {\bf 8} (1959), 109--120.

\bibitem{S99}
R. \v{S}krekovski. List improper coloring of planar graphs. \emph{Comb. Prob. Comp.}, {\bf 8} (1999), 293Ð299.


\bibitem{S76}
R. Steinberg,  The state of the three color problem. Quo Vadis, Graph Theory?, \emph{Ann. Discrete
Math.} {\bf 55} (1993), 211--248.


\bibitem{X08}
B. Xu, On $(3,1)^*$-coloring of planar graphs, \emph{SIAM J. Disc. Math.}, {\bf 23} (2008), 205--220.


\end{thebibliography}
\end{document}